\documentclass[10pt,bezier]{article}
\usepackage{amsmath,amssymb,amsfonts,graphicx,caption,subcaption}
\usepackage[colorlinks,linkcolor=red,citecolor=blue]{hyperref}

\newtheorem{prethm}{{\bf Theorem}}[section]
\newenvironment{thm}{\begin{prethm}{\hspace{-0.5
em}{\bf.}}}{\end{prethm}}
\newtheorem{prepro}{{\bf Theorem}}

\newtheorem{precor}[prethm]{{\bf Corollary}}

\newtheorem{preconj}[prethm]{{\bf Conjecture}}
\newenvironment{conj}{\begin{preconj}{\hspace{-0.5
em}{\bf.}}}{\end{preconj}}
\newtheorem{preremark}[prethm]{{\bf Remark}}

\newtheorem{prelem}[prethm]{{\bf Lemma}}

\newtheorem{preque}[prethm]{{\bf Problem}}
\newenvironment{prob}{\begin{preque}{\hspace{-0.5
em}{\bf.}}}{\end{preque}}

\newtheorem{prealphthm}{{\bf Problem}}

\newtheorem{preobserv}[prethm]{{\bf Observation}}

\newtheorem{predef}[prethm]{{\bf Definition}}

\newtheorem{preproposition}[prethm]{{\bf Proposition}}

\newtheorem{preproof}{{\bf Proof.}}
\newtheorem{preprooff}{{\bf Proof}}

\newenvironment{proof}[1]{\begin{preproof}{\rm
#1}\hfill{$\Box$}}{\end{preproof}}

\newtheorem{preproofF}{{\bf Proof of}}

\title{\bf\Large 
On the existence of minimally tough graphs having large minimum degrees
}

\author{{\normalsize{\sc Morteza Hasanvand${}$} }\vspace{3mm}
\\{\footnotesize{${}$\it Department of Mathematical
 Sciences, Sharif
University of Technology, Tehran, Iran}}
{\footnotesize{}}\\{\footnotesize{ $\mathsf{morteza.hasanvand@alum.sharif.edu }$ }}}

\date{}

\begin{document}
\maketitle
\begin{abstract}{
Kriesel conjectured that every minimally $1$-tough graph has a vertex with degree precisely $2$. Katona and Varga (2018) proposed a generalized version of this conjecture which says that every minimally $t$-tough graph has a vertex with degree precisely $\lceil 2t\rceil$, where $t$ is a positive real number. This conjecture has been recently verified for several families of graphs. For example, Ma, Hu, and Yang (2023) confirmed it for claw-free minimally $3/2$-tough graphs. Recently, Zheng and Sun (2024) disproved this conjecture by constructing a family of $4$-regular graphs with toughness approaching to $1$.

In this paper, we disprove this conjecture for planar graphs and their line graphs. In particular, we construct an infinite family of minimally $t$-tough non-regular claw-free graphs with minimum degree close to thrice their toughness. This construction not only disproves a renewed version of Generalized Kriesel's Conjecture on non-regular graphs proposed by Zheng and Sun (2024), it also gives a supplement to a result due to Ma, Hu, and Yang (2023) who proved that every minimally $t$-tough claw-free graph with $t\ge 2$ has a vertex of degree at most $3t+ \lceil (t-5)/3\rceil$. Moreover, we conjecture that there is not a fixed constant $c$ such that every minimally $t$-tough graph has minimum degree at most $\lceil c t \rceil$.
} {\small
}
\end{abstract}
%
%
%
%
%
%
%
%
%
%
\section{Introduction}
In this article, all graphs are considered simple. Let $G$ be a graph. 
The vertex set, the edge set, the number of components, and the independence number of $G$ are denoted by $V(G)$, $E(G)$, $\omega(G)$, $\alpha(G)$, respectively. For a positive real number $t$, a graph $G$ is said to be {\it $t$-tough} if $\omega(G)\le \max\{1,\frac{1}{t}|S|\}$ for all $S\subseteq(G)$.
The maximum positive real number $t$ such that $G$ is $t$-tough is called the toughness of $G$ that is denoted by $t(G)$.
In another word, for a non-complete graph $G$, $t(G)=\min\{\frac{|S|}{\omega(G\setminus S)}: \omega(G\setminus S) \ge 2\}$.
Likewise, $G$ is said to be {\it minimally $t$-tough}, if $G-e$ is not $t$-tough, for every edge $e\in E(G)$. 
A graph $G$ is called {\it claw-free}, if it does not an induced star of size three. 
A {\it circulant graph} $\mathcal{C}(n, A)$ is a graph $G$ with vertices $v_1,\ldots, v_n$ and two vertices $v_i$ and $v_j$ are adjacent if and only if $i-j\in A\subseteq \mathbb{Z}_n$ 
(mod $n$), where $\mathbb{Z}_n$ denotes the cyclic group of order $n$ and
 $\mathbb{Z}_n=\{1,\ldots, n\}$. 
A graph $H$ is called {\it $s$-solid}, if it is obtained from a graph $G$ by replacing every vertex $v$ with $s(v)$ copies $v_1,\ldots, v_{s(v)}$, and two vertices $v_i$ and $w_j$ are adjacent if $v$ and $w$ are adjacent in the original graph $G$, 
where $s$ is a positive integer-valued function on $V(G)$; two $2$-solid graphs are drawn in Figure~\ref{Fig:solid-graphs}. 
We denote this graph by $\mathcal{S}(G, s)$.
We denote by $S(G)$ the {\it subdivision graph} of $G$ which can be obtained from it by inserting a new vertex on each edge.
 For a graph $G$, the {\it line graph} $L(G)$ is a graph whose vertex set is $ E(G)$ and also two $e_1,e_2\in E(G)$ are adjacent in $L(G)$ if they have a common end in $G$. 
The {\it square} of a graph $G$ is a graph with the same vertex set and two vertices are adjacent if their distance in $G$ is at most $2$. For two graphs $G$ and $H$, the {\it Cartesian product} $G\square H$ refers to the graph with the vertex set $\{(v, w): v\in V(G), w\in V(H)\}$ such that $(v,w)$ and $(v', w')$ are adjacent if $vv'\in E(G)$ or $ww'\in E(H)$. 
The complete graph and the path graph of order $n$ are denoted by $K_n$ and $P_n$, respectively.
A graph is called {\it $r$-regular}, if all vertex-degrees are the same integer number $r$.
 A graph $G$ of order at least $k+1$ is called $k$-connected, if it remains connected after removing any set of vertices of size at most $k-1$.

In 1973 Chv\' atal~\cite{Chvatal-1973} introduced the concept of toughness inspired by a property of Hamiltonian graphs.
In fact, every Hamiltonian graph must be $1$-tough.
He also conjectured that tough enough graphs admit a Hamiltonian cycle. 
This concept is stronger than connectivity, as he noted that every $t$-tough graph is $2t$-connected as well.

\begin{thm}{\rm (\cite[Proposition 1.3]{Chvatal-1973})}
{Every $t$-tough graph is also $2t$-connected.
}\end{thm}

For higly connected graphs, Mader (1971) \cite{Mader-1971} proved that every minimally $k$-connected graph has a vertex of degree $k$. Broersma, Engsberg, Trommel (1999) \cite{Broersma-Engsberg-Trommel-1999} defined the concept of minimally tough graphs. Motivated by Mader's result, one may ask whether minimally tough graphs have small minimum degrees (even in special families of graphs). Kriesell conjectured that every minimally $1$-tough graph has minimum degree precisely two (similar to Hamiltonian cycles). 
\begin{conj}{\rm (\cite{Kaiser-web})}\label{intro:conj:Kriesell}
{Every minimally $1$-tough graph has a vertex of degree $2$.
}\end{conj}

A generalization of this conjecture is also proposed by Katona and Varga (2018)~\cite{Katona-Varga-2018} which says that minimally tough graph have minimum degree close to twice their toughnesses.
\begin{conj}{\rm (\cite{Katona-Varga-2018})}\label{intro:conj:Generalized}
{Every minimally $t$-tough graph has a vertex of degree $\lceil 2t\rceil $, where $t>0$.
}\end{conj}

Recently, Zheng and Sun (2024) \cite{Zheng-Sun-2024} showed that Conjecture~\ref{intro:conj:Generalized} fails for real numbers $t$ close enough to $1$. As a consequence, they also remarked that Conjecture~\ref{intro:conj:Kriesell} is sharp according to their construction (the necessary toughness cannot be increased even a little).

\begin{thm}{\rm (\cite{Zheng-Sun-2024})}
{For every inetegr $k$ with $k\ge 2$, there is a $4$-regular minimally $(1+\frac{1}{2k-1})$-tough graph.
}\end{thm}

They also proposed the following revised version of Conjecture~\ref{intro:conj:Generalized} for non-regular graphs and posed one open problem for further investigation in regular graphs. Their problem is recently settled by Cheng, Li, and Liu (2024) \cite{Cheng-Li-Liu-2024} who disproved Conjecture~\ref{intro:conj:Generalized} by a new family of $4$-regular graphs and $6$-regular graphs.

\begin{conj}{\rm (\cite{Zheng-Sun-2024})}\label{intro:conj:renewed}
{Every non-regular minimally $t$-tough graph has a vertex of degree $\lceil 2t\rceil $, where $t>0$.
}\end{conj}

In this paper, we first present all counterexamples of Conjecture~\ref{intro:conj:Generalized} having at most eleven vertices. 
Among all of them, one graph was already discovered in \cite{Zheng-Sun-2024}, and one of them is not regular which shows that Conjecture~\ref{intro:conj:renewed} fails for small graphs.
Next, we show that Generalized Kriesell's Conjecture fails even in $4$-regular planar graphs and the square of planar graphs
(by calling a special counterexample in \cite{List-coloring}).
Note that Generalized Kriesel's Conjecture seems true for a large ratio of graphs and it has been recently verified for several families of graphs, see \cite{Katona-Soltesz-Varga-2018, Katona-Khan-2024, Ma-Hu-Yang-2024}). 
In particular, Ma, Hu, and Yang (2023) confirmed it for claw-free minimally $3/2$-tough graphs. 

\begin{thm}{\rm (\cite{Ma-Hu-Yang-2023})}
{Every claw-free minimally $3/2$-tough graph has minimum degree $3$.
}\end{thm}

In general, Ma, Hu, and Yang (2023) \cite{Ma-Hu-Yang-2023-arXiv-claw-free} confirmed a weaker version of it on claw-free graphs by giving the explicit linear bound $3t+ \lceil \frac{t-5}{3}\rceil$ on the minimum degrees. In Section~\ref{sec:3t}, disprove Generalized Kriesel's Conjecture for claw-free graphs by constructing a family of (non-regular) counterexamples whose minimum degrees are close to thrice their highnesses. It remains to decide how much the number $\lceil \frac{t-5}{3}\rceil $ in the following theorem can be reduced.

\begin{thm}{\rm (\cite{Ma-Hu-Yang-2023})}\label{intro:thm:general-claw-free}
{Every minimally $t$-tough claw-free graph with $t\ge 2$ has a vertex of degree at most $3t+ \lceil \frac{t-5}{3}\rceil$. 
}\end{thm}

Motivated by these new graph constructions, one may ask whether every minimally $t$-tough graph has minimum degree at most
 $\lceil c t \rceil$ for a constant number $c$. We conjecture that the answer is false and put forward a conjecture in Section~\ref{sec:revised-conj} for this purpose. 
%
%
%
%
%
%
%
%
\section{Small graphs and planar graphs}
\label{sec:small-graphs}
By a computer search, we observed that among all connected graphs of order at most $10$, there is a unique counterexample of Conjecture~\ref{intro:conj:Generalized} which was already found in \cite[Theorem~9]{Zheng-Sun-2024} (see the left graph is Figure~\ref{small:graphs}) and there are only three counterexamples of order $11$ (the middle-right graph was also found in 
\cite[Theorem~1.6]{Cheng-Li-Liu-2024}). One of these graphs is not regular (the right graph) and consequently Conjecture~\ref{intro:conj:renewed} fails. (We will construct an infinite family of counterexamples in Theorems~\ref{thm:claw-free} and~
\ref{thm:3t-1}). Note that there is another counterexample which is shown in the right graph of Figure~\ref{Fig:solid-graphs}. 

\begin{figure}[h]
 \centering
 \includegraphics[scale = 1.4]{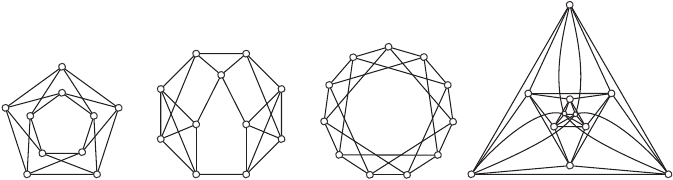}
 \caption{All minimally $t(G)$-tough graphs $G$ of order at most $11$ satisfying $\delta(G) > \lceil 2t(G)\rceil $, where 
 $t(G)\in \{4/3, 3/2, 6/4, 5/2\}$, respectively (from left to right).} 
\label{small:graphs}
\end{figure}
In the following theorem, we are going to introduce planar $4$-regular counterexamples using a simple structure based on the copy of the complete graph $K_5$ minus an edge. We originally found the smallest graph by applying a computer search on the restricted outputs of the program {\it plantri} due to Brinkmann and McKay (2007) \cite{Brinkmann-McKay-2007}.

\begin{thm}
{There are infinitely many planar minimally $3/2$-tough $4$-regular graphs.
}\end{thm}
\begin{proof}
{Let $m$ be an even number with $m\ge 4$.
For each $i$ with $1\le i\le m$, let $v_{i, 1}, \ldots, v_{i, 6}$ be six numbers.
First, we add five edges $v_{i, 1}v_{i, 2},\ldots, v_{i, 4}v_{i, 5}, v_{i, 5}v_{i, 1}$ (which forms a cycle).
Next, we add four edges $v_{i, 6}v_{i, j}$, where $j\in \{1,\ldots, 4\}$.
Finally, if $i$ is even, we add three edges $v_{i, 1}v_{i+1, 1},v_{i, 2}v_{i+1, 5}, v_{i, 5}v_{i+1, 2}$,
and if $i$ is odd, we add three edges $v_{i, 4}v_{i+1, 4},v_{i, 3}v_{i+1, 5}, v_{i, 5}v_{i+1, 3}$ 
(if $i=m$, we compute $i+1$ modulo $m$ that means $i+1=1$). 
Call the resulting graph $G$. Obviously, this graph is planar and $4$-regular.
 For the smallest case $m$=4, the graph $G$ is illustrated in Figure~\ref{fig:planar-graph}.

\begin{figure}[h]
 \centering
 \includegraphics[scale = 2.5]{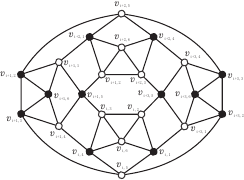}
 \caption{A minimally $3/2$-tough planar $4$-regular graph of order $24$.} 
\label{fig:planar-graph}
\end{figure}

We claim that $t(G)= 3/2$.
We first show that $t(G)\ge 3/2$.
Let $S$ be an arbitrary cut-set of $G$.
Define $G_i$ to be the induce subgraph of $G$ with the vertex set $V_i=\{v_{i, j}:1\le j\le 6\}$ and let $S_i=S\cap V_i$.
Since $G_i$ has independence number at most two, we must have $ \omega(G_i\setminus S_i)\le 2$. 
Let $\omega_i$ be the number of components of $G_i\setminus S_i$ 
which is not connected to a component of $G_{i+1}\setminus S_{i+1}$ in $G\setminus S$. 
Define $\theta_i=|S_i|-\frac{3}{2}\omega_i$. Let start with the following simple but useful subclaims.

{\bf Subclaim 1}: If $G_i\setminus S_i$ is disconnected, then $|S_i|\ge 3$ or $S_{i}=\{v_{i,1}, v_{i,4}\}$

{\bf Subclaim 2}: If $\{v_{i,1}, v_{i,5} \} \subseteq S_i$ or $\{v_{i,4}, v_{i,5} \} \subseteq S_i$, then $\theta_i \ge 1/2$.
In particular, in both cases, we have $\theta_i \ge 1$, if $|S_i|\ge 3$ or there is one component of $G_i\setminus S_i$ connecting to a component of 
$G_{i+1}\setminus S_{i+1}$.

If there is a component of $G$ having one vertex from every subgraph $G_i$, then 
$\omega(G\setminus S)\le \sum_{1\le i\le m} (\omega(G_i\setminus S_i)-1) +1$.
Moreover, $\omega_i\le 1$ and $\theta_i\ge 0$ for all indices $i$.
Since $S$ is a cut-set, there is a subgraph $G_i$ in which $G_i\setminus S_i$ is disconnected.
If $|S_i|\ge 3$, then $\theta_i \ge 3/2$. Otherwise, by Subclaim 1, $S_i=\{v_{i,1}, v_{i,4}\}$.
If there is a subgraph $G_j$ in which $G_j\setminus S$ is connected and $|S_j|>0$, then $\theta_i + \theta_j \ge 1/2 + 1=3/2$.
Otherwise, there must be at least two other subgraphs $G_{j_1}$ and $G_{j_2}$ in which
 $G_{j_t}\setminus S_{j_t}$ is connected. Thus $\theta_i + \theta_{j_1}+\theta_{j_2} \ge 1/2 + 1/2+1/2=3/2$.
Therefore, $\sum_{1\le i\le m}\theta_i \ge 3/2$ which implies that $|S|\ge \frac{3}{2} \omega(G\setminus S)$.
Now, assume that there is no component of $G\setminus S$ having one vertex from every subgraph $G_i$.
Therefore, 
$$\sum_{1\le i\le m}\theta_i=\sum_{1\le i\le m}(|S_i|-\frac{3}{2}\omega_i)=
|S|-\frac{3}{2}\omega(G\setminus S).$$
We say an index $i$ is {\it good} if $\theta_i\ge 0$, and {\it bad} otherwise.
If all indices are good, then $\frac{|S|}{\omega(G\setminus S)}\ge 3/2$.
Otherwise, by applying the following assertion, one can conclude that there is a partition $P$ of indices $\{1,\ldots, m\}$ such that for every $I\in P$, $\sum_{i\in I}\theta_i\ge 0$ and so $\sum_{1\le i\le m}\theta_i\ge 0$ and $\frac{|S|}{\omega(G\setminus S)}\ge 3/2$.

{\bf Subclaim 3}: If $b$ is a bad index, then there exists an index $j\in \{1,2,3\}$ such that $\sum_{b\le i\le b+j}\theta_i\ge 0$ and all indices $b+1,\ldots, b+j$ are good.

{\bf Proof of Subclaim 3}. By the isomorphic property, we may assume that $b$ is odd.
Suppose the assertion false. This implies that $\theta_b +\max\{0, \theta_{b+1},\theta_{b+1}+\theta_{b+2}, \theta_{b+1}+\theta_{b+2}+\theta_{b+3}\} \le -1/2$. In addition, any component of $G_b\setminus S_b$ is not connected to a component of $G_{b+1}\setminus S_{b+1}$.
Obviously, $\omega_b\in \{1,2\}$. To derive a contradiction, we shall consider the following two cases:

{\bf Case A}: $\omega_b=1$.

In this case, $|S_b|\le 1$ and $-3/2 \le \theta_{b}\le -1/2$.
If $ S_{b} \cap \{v_{b, 3}, v_{b, 4}, v_{b, 5} \} =\emptyset$,
 then $\{v_{b+1, 3}, v_{b+1, 4}, v_{b+1, 5} \}\subseteq S_{b+1}$ and $\theta_{b+1}\ge 3/2$.
Since $\theta_b+\theta_{b+1}\ge 0$, we derive a contradiction. Assume that $ |S_{b}\cap \{v_{b, 3}, v_{b, 4}, v_{b, 5}\}| = 1$.
In this case, $\theta_{b}= -1/2$. 
If $S_{b}=\{v_{b, 5}\}$, then since the unique component of $G_b\setminus S_b$ is not connected to a component of $G_{b+1}\setminus S_{b+1}$, we must have $\{v_{b+1, 4}, v_{b+1, 5}\}\subseteq S_{b+1}$ and $\theta_{b+1}\ge 1/2$ using Subclaim 2.
Since $\theta_b+\theta_{b+1}\ge 0$, we derive a contradiction.
If $ S_{b}=\{v_{b, 4}\}$, then it is not difficult to check that 
$S_{b+1}=\{v_{b+1, 3}, v_{b+1, 5}, v_{b+1, 6}\} $ and $\theta_{b+1}= 0$. 
Similarly, since any component of $G_{b+1}\setminus S_{b+1}$ is not connected to a component of $G_{b+2}\setminus S_{b+2}$,
we must have $\{v_{b+2, 1}, v_{b+2, 5} \}\subseteq S_{b+2}$ and $\theta_{b+2}\ge 1/2$ using Subclaim 2.
Since $\theta_b+\theta_{b+1}+\theta_{b+2}\ge 0$, we derive a contradiction.
If $ S_{b}=\{v_{b, 3}\}$, then it is not difficult to check that 
$ S_{b+1}=\{v_{b+1, 1}, v_{b+1, 3}, v_{b+1, 4}\}$ and $\theta_{b+1}= 0$. 
Since any component of $G_{b+1}\setminus S_{b+1}$ is not connected to a component of $G_{b+2}\setminus S_{b+2}$, we must have $S_{b+2}=\{v_{b+2, 2}, v_{b+2, 5}, v_{b+2,6} \} $ and $\theta_{b+2}= 0$.
Therefore, we must have $S_{b+3}=\{v_{b+3, 4}, v_{b+3, 5}\} $ and $\theta_{b+2}\ge 1/2$ using Subclaim 2.
Since $\theta_b+\theta_{b+1}+\theta_{b+2}+\theta_{b+3}\ge 0$, we again derive a contradiction.

{\bf Case B}: $\omega_b=2$.

In this case, by Subclaim 1, we must have $S_b=\{v_{b, 1}, v_{b, 4}\}$ and so $\theta_b= -1$. 
Since there is not a component of $G_b\setminus S_b$ connecting to a component of $G_{b+1}\setminus S_{b+1}$, 
it is not difficult to check that $\{v_{b+1, 3}, v_{b+1, 5}\}\subseteq S_{b+1}$, $\{v_{b+1, 1}, v_{b+1, 2}\}\cap S_{b+1}=\emptyset$, and $\theta_{b+1}\ge 0$.
Consequently, $S_{b+2}=\{v_{b+2, 1}, v_{b+2, 2} \}$ and $\theta_{b+2}= 1/2$ using Subclaim 2.
This implies that $S_{b+3}=\{v_{b+3, 4}, v_{b+3, 5} \}$ and $\theta_{b+3}= 1/2$ using Subclaim 2 again.
Since $\theta_b+\theta_{b+1}+\theta_{b+2}+\theta_{b+3}\ge 0$, we derive contradiction, as desired.

Therefore, $t(G) \ge 3/2$.
From now on, for notational simplicity, for each odd index $i$, we set $S_i = \{v_{i, 1}, v_{i, 4}\} $, and for even index $i$, 
we set $S_i = V_i\setminus \{v_{i, 1}, v_{i, 4}\}$; recall that $V_i=\{v_{i, j}: 1\le j\le 6\}$.
If we set $S=\cup_{1\le i\le m} S_i$, then it is easy to check that $|S|=3m$ and $\omega(G\setminus S) = 2m$. 
This implies that $t(G)\le \frac{|S|}{\omega(G\setminus S)}=3/2$. Hence $t(G) = 3/2$ and the claim is proved.
It remains to show that $G$ is minimally $t(G)$-tough. 
Let $e$ be an arbitrary edge of $G$. We shall consider the following cases:

{\bf Case 1}: Both ends of $e$ lie in the set $\{v_{i,2}, v_{i,3}, v_{i,6}\}$.

By the isomorphic property, we may assume that $i$ is odd.
Let $v$ be the unique vertex in $\{v_{i,2}, v_{i,3}, v_{i,6}\}$ not incident with $e$.
If we set $S=(\cup_{1\le j\le m}S_j )\cup \{v\}$, then it is easy to check that $|S|=3m+1$ and $\omega((G-e)\setminus S) = 2m+1$. 
This implies that $t(G-e)\le \frac{|S|}{\omega(G\setminus S)}=\frac{3m+1}{2m+1}<3/2$.

{\bf Case 2}: $e\in \{v_{i,1}v_{i,6}, v_{i,2}v_{i,1}, v_{i,2}v_{i+1,5}\}$.

By the isomorphic property, we may assume that $i$ is even.
If $e=v_{i,1}v_{i,6}$, then we set $S=(\cup_{1\le j\le m}S_j)\setminus \{v_{i,6}\}$.
Otherwise, we set $S=(\cup_{1\le j\le m}S_j)\setminus \{v_{i,2}\}$.
In both cases, it is easy to check that $|S|=3m-1$ and $\omega((G-e)\setminus S) = 2m$. 
This implies that $t(G-e)\le \frac{|S|}{\omega(G\setminus S)}=\frac{3m-1}{2m}<3/2$.

{\bf Case 3}: $e=v_{i,4}v_{i+1,4}$.

By the isomorphic property, we may assume that $i$ is odd.
If we set 
$S=((\cup_{1\le j\le m}S_j ) \setminus \{v_{i-1,2}, v_{i,4}, v_{i+1, 3}\} ) \cup
 \{v_{i, 5}\}$, then 
it is not difficult to check that $|S|=3m-2$ and $\omega((G-e)\setminus S) = 2m-1$. 
This implies that $t(G-e)\le \frac{|S|}{\omega(G\setminus S)}=\frac{3m-2}{2m-1}<3/2$.

{\bf Case 4}: $e=v_{i,1}v_{i,5}$.

By the isomorphic property, we may assume that $i$ is even.
If we set $S=((\cup_{1\le j\le m}S_j) \setminus \{v_{i-2, 2}, v_{i-1,4}, v_{i,3}, v_{i, 5}\} )\cup 
\{v_{i-1,3}, v_{i-1,6}, v_{i-1,5}, v_{i,4}, v_{i+1,2}\}
 ,$ then 
it is not difficult to check that $|S|=3m+1$ and $\omega((G-e)\setminus S) = 2m+1$. 
This implies that $t(G-e)\le \frac{|S|}{\omega(G\setminus S)}=\frac{3m+1}{2m+1}<3/2$.

The remaining edges are similar with respect to isomorphic property of the graph $G$.
Hence the proof is completed.
}\end{proof}

In 1999 Broersma, Engsberg, and Trommel \cite{Broersma-Engsberg-Trommel-1999} studied minimally tough square graphs. 
In particular, they showed that if the square of a graph is minimally $2$-tough, then the original graph must be minimally $2$-connected and triangle-free. We have recently used a small square graph in \cite{List-coloring} to provide a solution for several conjectures and open problems related to list coloring. Surprisingly, we observed that this graph is another counterexample to Conjecture~\ref{intro:conj:Generalized}. (Note that there is an infinite family of such examples, but this is enough for our purpose). 

\begin{thm}
{There is a minimally $3$-tough $7$-regular graph $G$ of order $12$.
In particular, $G$ is the square of the planar line graph $L(S(K_4))$.
}\end{thm}
\begin{proof}
{Let $H$ be the cubic graph illustrated in Figure~\ref{Fig:12-special-square-graph} and 
let $G$ be the square of $H$. 
It is not difficult to show that $G$ contains only four maximum independent sets of size $4$
 which are shown in Figure~\ref{Fig:12-special-square-graph} by numbers $1,\ldots, 4$.

\begin{figure}[h]
 \centering
 \includegraphics[scale = 1.75]{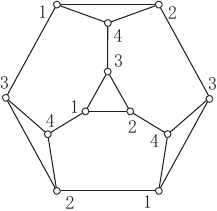}
 \caption{The square of the planar graph $H$ is a minimally $3$-tough $7$-regular graph.} 
\label{Fig:12-special-square-graph}
\end{figure}

Let $I$ be a maximum independent set of $G$. If we set $S=V(G)\setminus I$, then 
$t(G)\le \frac{|S|}{\omega(G\setminus S)}= \frac{|S|}{|I|}=\frac{9}{3}=3$.
Let $S$ be a cut-set of $V(G)$. 
Since $G$ is $7$-connected, we must have $|S|\ge 7$.
If $\omega(G\setminus S)=2$, then 
$ \frac{|S|}{\omega(G\setminus S)}\ge \frac{7}{2}> 3$.
If $\omega(G\setminus S)\ge 3$, then 
according to the property of $G$, the subgraph $G\setminus S$ must consist of three isolated vertices and so 
$S=|V(G)|-3=9$ which implies that $ \frac{|S|}{\omega(G\setminus S)}= 3$.
Consequently, $t(G)=3$.

We are going to show that $G$ is minimally $t(G)$-tough. Let $e=uv$ be an arbitrary edge of $G$.
If $e\in E(H)$, then it is not difficult to check that there is a unique edge $e'=u'v'$ of $H$ which is not connected to $e$ by an edge in $G$. If we set $S=V(G)\setminus \{u, v, u', v'\}$, then 
$t(G-e) \le \frac{|S|}{\omega((G-e)\setminus S)}=\frac{8}{3}<3=t(G)$.
If $e\not \in E(H)$, then we may assume that $u$ and $v$ have a common neighbor $w$ in $H$ for which $w$ and $v$ lie in a triangle in $H$. In this case, we again have $t(G-e) \le \frac{|S|}{\omega((G-e)\setminus S)}=\frac{8}{3}<3=t(G)$, where 
$S=V(G)\setminus (I\cup \{u\})$ and $I$ is the maximum independent set of $G$ including $v$.
Hence the proof is completed.
}\end{proof}
%
%
%
%
%
%
%
%
%
\section{Minimally tough claw-free graphs with minimum degree close to thrice their toughness}
\label{sec:3t}
In the following theorem, we show that the upper bound in Theorem~\ref{intro:thm:general-claw-free} cannot be replaced by $3t-2$ (by setting $2t=m=\lceil \frac{2n+1}{3}\rceil=\frac{2n+1}{3}$ provided that $n\stackrel{3}{\equiv}1$). Note that all graphs $G$ presented below are not regular, and alternatively disprove Conjecture~\ref{intro:conj:renewed}. In addition, $G$ is the line graph of the graph obtained from two copies of $n$-stars by identifying $m$ pairs of pendant vertices belonging to different stars.
\begin{thm}\label{thm:claw-free}
{Let $n$ and $m$ be two positive integers satisfying $\frac{2}{3}n<m < n$ and $n\ge 7$. If $G$ is the graph obtained from the Cartesian product $K_n\square P_2$ by deleting $m$ independent edges connecting two copies of $K_n$, then $G$ is a minimally $\frac{m}{2}$-tough claw-free graph satisfying $$\delta(G)=\Delta(G)-1=n-1.$$
}\end{thm}
\begin{proof}
{Let $G_i$ be a copy of the complete graph $K_n$ with vertices $v_{i,1},\ldots, v_{i,n}$, where $i=1,2$.
For each $j$ with $1\le j\le m$, we
add an edge $v_{1, j}v_{2, j}$ and call the resulting graph $G$. The special case $n=7$ and $m=4$ is illustrated in Figure~\ref{KnP2} (the right graph).

\begin{figure}[h]
 \centering
 \includegraphics[scale = 2.2]{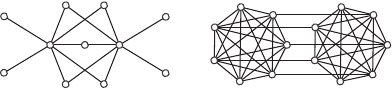}
 \caption{The line graph of the left graph is a minimally $5/2$-tough graph $G$ satisfying $\delta(G)=6$ (which is shown in the right).} 
\label{KnP2}
\end{figure}

If we set $S=\{v_{1, p}: 1\le p\le m\}$, 
then $\omega(G\setminus S)=2$ and $|S|=m$ which implies that $t(G)\le \frac{|S|}{\omega(G\setminus S)} =\frac{m}{2}$.
On the other hand, if $S$ is a cut-set of $G$, then $\omega(G\setminus S)=2$ and $|S|\ge m$,
because $\alpha(G)=2$ and $G$ is $m$-connected.
This implies that $ \frac{|S|}{\omega(G\setminus S)}\ge \frac{m}{2}$ and so $t(G)= \frac{m}{2}$.
We are going to show that $G$ is minimally $t(G)$-tough. Let $e$ be an arbitrary edge of $G$.
If $e=v_{1, j}v_{2,j}$, then $t(G-e) \le \frac{|S|}{\omega((G-e)\setminus S)}=\frac{m-1}{2}<\frac{m}{2}=t(G)$, 
where $S=\{v_{1, p}: 1\le p\le m \text { and } p\neq j\}$.
If $e=v_{i, j_1}v_{i,j_2}$, then $t(G-e) \le \frac{|S|}{\omega((G-e)\setminus S)}=\frac{n}{3}<\frac{m}{2}=t(G)$, where 
$S=(V(G_{i})\setminus \{v_{i, j_1},v_{i,j_2}\})\cup \{v_{3-i, j_1},v_{3-i,j_2}\}$.
This completes the proof.
}\end{proof}
In the following theorem, we construct two new families of non-regular and regular minimally $t$-tough graphs with minimum degree at least $\lceil 3t\rceil-1$. 
\begin{thm}\label{thm:3t-1}
{If $n$ is an integer with $n\ge 3$, then the Cartesian product $G=K_n\square P_3$ is a non-regular minimally $\frac{n+1}{3}$-tough graph satisfying $$\delta(G)=\Delta(G)-1=n.$$
 In addition, if $G_0$ is the graph obtained from $G$ be deleting the middle vertex of an induced path $v_1v_2v_3$ and replacing the new edge $v_1v_3$, then $G_0$ is $n$-regular and minimally $\frac{n+1}{3}$-tough.
}\end{thm}
\begin{proof}
{Let $G_i$ be a copy of the complete graph $K_n$ with vertices $v_{i,1},\ldots, v_{i,n}$, where $i=1,2,3$.
For each $p$ with $1\le p\le n$, we
insert two edges $v_{1, p}v_{2, p}$ and $v_{2, p}v_{3, p}$ into these graphs and call the resulting graph $G$.
This graph is isomorphic to the Cartesian product $ K_n\square P_3$.
The special case $n=5$ is illustrated in Figure~\ref{Fig:KnP3} (the left graph). 
\begin{figure}[h]
 \centering
 \includegraphics[scale = 1]{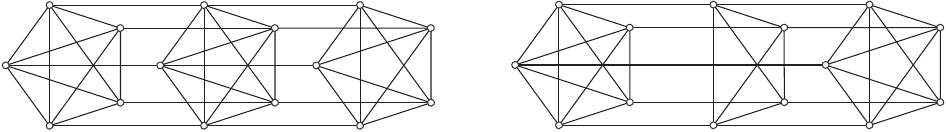}
 \caption{Two minimally $2$-tough graphs with minimum degree $5$ (regular and non-regular).} 
\label{Fig:KnP3}
\end{figure}
We claim that $t(G)= \frac{n+1}{3}$. If we set $S=\{v_{2,p}: 1\le p< n\} \cup \{v_{1,n}, v_{3, n}\}$, 
then $|S|=n+1$ and $\omega(G\setminus S)=3$. 
Hence $t(G)\le \frac{|S|}{\omega(G\setminus S)}=\frac{n+1}{3}$. 
It remains to show that $t(G)\ge \frac{n+1}{3}$.
Let $S$ be an arbitrary cut-set of $V(G)$.
 If $|S|\ge n+1$, then
 $\frac{|S|}{\omega(G\setminus S)}\ge \frac{|S|}{\alpha(G\setminus S)}\ge \frac{n+1}{3}$.
Assume that $|S|\le n$. According to the construction of $G$, we must have $\omega(G\setminus S)=2$.
More precisely, $|S\cap \{v_{1,p}, v_{2,p}, v_{3,p}\} |= 1$ for all integers $p$ with $1\le p\le n$.
Therefore, $\frac{|S|}{\omega(G\setminus S)}= \frac{n}{2}\ge \frac{n+1}{3}$. 
Hence the claim is proved.
We are going to show that $G$ is minimally $t(G)$-tough. 
Let $e$ be an arbitrary edge of $G$.
If $e=v_{2, j}v_{i,j}$, then $t(G-e) \le \frac{|S|}{\omega((G-e)\setminus S)}=\frac{n}{3}<\frac{n+1}{3}=t(G)$, 
where $S=\{v_{2,p}:1\le p\le n \text{ and } p\neq j\}\cup \{v_{4-i, j}\}$.
If $e=v_{i, j_1}v_{i,j_2}$ and $i\neq 2$, then $t(G-e) \le \frac{|S|}{\omega((G-e)\setminus S)}=\frac{n}{3}<\frac{n+1}{3}=t(G)$, 
where $S=\{v_{i,p}:1\le p\le n \text{ and } p\neq j_1, j_2\}\cup \{v_{2, j_1},v_{2,j_2}\}$.
If $e=v_{2, j_1}v_{2,j_2}$, then $t(G-e) \le \frac{|S|}{\omega((G-e)\setminus S)}=\frac{n+2}{4}<\frac{n+1}{3}=t(G)$, where $S=\{v_{2,p}:1\le p\le n \text{ and } p\neq j_1, j_2\}\cup \{v_{1, j_1},v_{1,j_2}, v_{3, j_1},v_{3,j_2}\}$.
This completes the proof of the first part.

Let $G_0$ be the graph obtained from $G$ by deleting the vertex $v_{2,n}$ and adding the new edge $v_{1,n}v_{3,n}$.
Obviously, $G_0$ is $n$-regular. The special case $n=5$ is illustrated in Figure~\ref{Fig:KnP3} (the right graph). 
For notional simplicity, let us use $H$ instead of $G_0$.
We claim that $t(H)= \frac{n+1}{3}$. 
If we set $S=\{v_{2, p}: 1\le p\le n-2\} \cup \{v_{1,n-1}, v_{3, n-1}, v_{3,n}\}$, 
then $|S|=n+1$ and $\omega(H\setminus S)=3$. Hence $t(H)\le \frac{|S|}{\omega(H\setminus S)}=\frac{n+1}{3}$. 
It remains to show that $t(H)\ge \frac{n+1}{3}$.
Let $S$ be an arbitrary cut-set of $V(H)$. 
 If $|S|\ge n+1$, then
 $\frac{|S|}{\omega(H\setminus S)}\ge \frac{|S|}{\alpha(H\setminus S)}\ge \frac{n+1}{3}$.
Assume that $|S|\le n$.
According to the construction of $H$, we must have $\omega(H\setminus S)=2$. More precisely,
 $|S\cap \{v_{1,n}, v_{3,n}\}| = 1$ and $|S\cap \{v_{1,p}, v_{2,p}, v_{3,p}\} |= 1$ 
for all integers $p$ with $1\le p\le n-1$.
Therefore, $\frac{|S|}{\omega(H\setminus S)}= \frac{n}{2}\ge \frac{n+1}{3}$. 
Hence the claim is proved.
We are going to show that $H$ is minimally $t(H)$-tough. 
Let $e$ be an arbitrary edge of $H$.
If $e=v_{1,n}v_{3,n}$, then $t(H-e) \le \frac{|S|}{\omega((H-e)\setminus S)}=\frac{n}{3}<\frac{n+1}{3}=t(H)$, 
where $S=\{v_{2,p}:1\le p \le n-2\}\cup \{v_{1, n-1}, v_{3, n-1}\}$.
If $e=v_{2, j}v_{i,j}$, then $t(H-e) \le \frac{|S|}{\omega((H-e)\setminus S)}=\frac{n}{3}<\frac{n+1}{3}=t(H)$, 
where $S=\{v_{2,p}:1\le p\le n-1 \text{ and } p\neq j\}\cup \{v_{4-i, j}, v_{1, n}\}$.
If $e=v_{i, j_1}v_{i,j_2}$ and $i\neq 2$, then $t(H-e) \le \frac{|S|}{\omega((H-e)\setminus S)}=\frac{n}{3}<\frac{n+1}{3}=t(H)$, 
where $S=\{v_{i,p}:1\le p\le n \text{ and } p\neq j_1, j_2\}\cup \{v_{2, j_1}, v\}$ in which $v=v_{2,j_2}$ when $n\not \in \{j_1, j_2\}$, and $v=v_{4-i,n}$ when $j_1 \neq j_2=n$.
If $e=v_{2, j_1}v_{2,j_2}$, then $t(H-e) \le \frac{|S|}{\omega((H-e)\setminus S)}=\frac{n+2}{4}<\frac{n+1}{3}=t(H)$, where $S=\{v_{2,p}:1\le p\le n-1 \text{ and } p\neq j_1, j_2\}\cup \{v_{1, j_1},v_{1,j_2}, v_{3, j_1},v_{3,j_2}, v_{1, n}\}$.
Hence the proof is completed.
}\end{proof}
%
%
%
%
%
%
%
%
%
%
%
\section{A revised conjecture}
\label{sec:revised-conj}

Motivated by Theorem~\ref{thm:3t-1}, one may ask whether every minimally $t$-tough graph has minimum degree at most
 $\lceil c t \rceil$ for a constant number $c$. We conjecture that the answer is false (even possibly in $2$-solid graphs)
and put forward the following conjecture. To support this conjecture, by applying a computer search on vertex-transitive graphs on up to $35$ vertices \cite{Royle-Holt}, we could discover a minimally $4$-tough $24$-regular (resp. $6$-tough $36$-regular) $2$-solid graph $G$ of order $70$ satisfying $\frac{\delta(G)}{t(G)}=6$.
\begin{conj}\label{conj:revised}
{For every positive real number $c$, there is a minimally $t(G)$-tough (possibly 2-solid) graph $G$ satisfying $\frac{\delta(G)}{t(G)}\ge c$.
}\end{conj}

The structure of those graphs were not easy to draw, but to see a more clear graph example with smaller ratio, one can consider the $2$-solid graph $G$ obtained from the circulant graph $\mathcal{C}(39,\{3,4\})$ for which $\frac{\delta(G)}{t(G)}=5.2$ (more precisely, $t(G)=40/26$ and $\delta(G)=8$). Note that this ratio for the left graph in Figure~\ref{Fig:solid-graphs} is $4.2$. In addition, the main result in \cite{Zheng-Sun-2024} says that this ratio for $2$-solid graphs obtained from odd cycles must tend to $4$. Fortunately, to compute efficiently the toughness of the $s$-solid graph $G$ obtained from a graph $H$, we only needed to consider cut-sets $S$ with the minimum $\frac{|S|}{\omega(G\setminus S)}$ and $|S|$ so that for every vertex $v\in V(H)$, $S$ either includes all copies of $v$ in $G$ or excludes all of them (together with using isomorphic properties of $G$).

\begin{figure}[h]
 \centering
 \includegraphics[scale = 0.78]{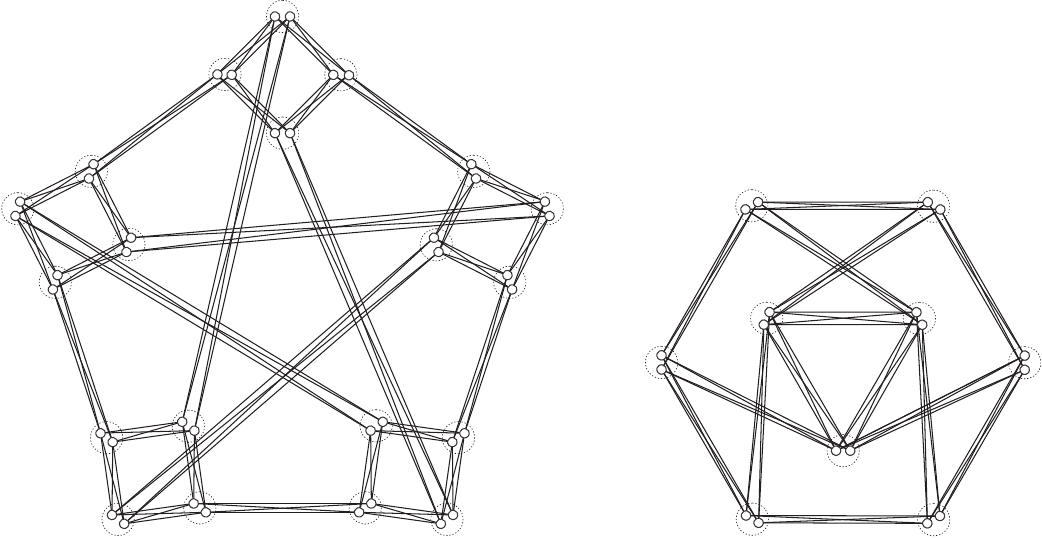}
 \caption{Two minimally $t(G)$-tough $2$-solid (regular and non-regular) graphs $G$ with minimum degree $6$ satisfying $t(G)\in \{\frac{20}{14},2\}$ and $\frac{\delta(G)}{t(G)}\in \{4.2, 3\}$.} 
\label{Fig:solid-graphs}
\end{figure}

Finally, we pose the following question for further investigation. Note that
Kriesel's Conjecture says that $f(1)$ is finite and it must be $2$ while Conjecture~\ref{conj:revised} says that,
regardless of $f(t)$ is finite or not, $\frac{f(t)}{t}$ must tend to infinity when $t$ tends to infinity.
\begin{prob}
{Let $t$ be a positive real number. 
Is it true that there exists a positive integer $f(t)$ such that every minimally $t$-tough graph $G$ satisfies $\delta(G)\le f(t)$?
}\end{prob}
%
%
%
%
%
%
%
%
%

\end{document}